\documentclass[11pt,reqno]{amsart}
\usepackage{iftex}
\ifPDFTeX
  \usepackage[utf8]{inputenc}
\fi
\usepackage[T1]{fontenc}
\IfFileExists{lmodern.sty}{\usepackage{lmodern}}{}
\usepackage{amsmath,amssymb,amsthm}
\usepackage{xcolor}
\usepackage[margin=25mm,top=28mm,bottom=28mm]{geometry}
\IfFileExists{microtype.sty}{\usepackage{microtype}}{}
\IfFileExists{enumitem.sty}{%
  \usepackage{enumitem}
  \setlist[enumerate]{leftmargin=*,label=(\roman*),itemsep=3pt}%
}{}
\usepackage[hidelinks,pdfencoding=auto]{hyperref}
\allowdisplaybreaks[2]
\numberwithin{equation}{section}
\theoremstyle{plain}
\newtheorem{theorem}{Theorem}[section]
\newtheorem{proposition}[theorem]{Proposition}
\newtheorem{lemma}[theorem]{Lemma}
\newtheorem{corollary}[theorem]{Corollary}
\theoremstyle{definition}
\newtheorem{definition}[theorem]{Definition}

\theoremstyle{remark}
\newtheorem{remark}[theorem]{Remark}
\newcommand{\R}{\mathbb R}
\newcommand{\Sn}{\mathbb S^n}
\newcommand{\Mp}{\mathcal M^+_{\lambda,\Lambda}}
\newcommand{\Mm}{\mathcal M^-_{\lambda,\Lambda}}
\newcommand{\Om}{\Omega}
\newcommand{\norm}[1]{\left\|#1\right\|}

\DeclareMathOperator{\tr}{tr}
\DeclareMathOperator{\diag}{diag}
\DeclareMathOperator{\dist}{dist}
\DeclareMathOperator{\diam}{diam}

\newcommand{\doi}[1]{\href{https://doi.org/#1}{doi: #1}}
\title[Boundary regularity for singular equations]
{Boundary   Regularity  
for Fully Nonlinear Singular Elliptic   Equations}

\author{Jiqi Dong}
\author{Xuemei Li}

\address{J. Dong: School of Mathematical Sciences, University of Jinan, Jinan, P.R.China 250022}
\address{X. Li: School of Mathematical Sciences, University of Jinan, Jinan, P.R.China 250022}

\email{17860357019@163.com}
\email{sms\_lixm@ujn.edu.cn}

\date{}

\begin{document}
\begin{abstract}
We study boundary regularity for the fully nonlinear singular elliptic equation of the form
$|Du|^\gamma F(D^2u)=f$, where $-1<\gamma<0$.
First, we prove pointwise boundary $C^{1,\alpha}$ estimates with pointwise
$C^{1,\alpha}$ Dirichlet data, which weakens the $C^2$ domain
smoothness  requirement from  Birindelli and
Demengel~\cite{BD10}.
The proof uses compactness and perturbation, without flattening the boundary.
Second, we extend the $W^{2,\delta}$ estimates of Li and Li~\cite{LiLi17} from balls to $C^{1,\alpha}$ domains
with $C^{1,\alpha}$ Dirichlet data{, for $0<\delta\le\delta_0<1$.
Here $\delta_0$ is an exponent in the interior $W^{2,\delta_0}$ estimate}.
The proof combines pointwise boundary $C^{1,\alpha}$ estimates,
interior $W^{2,\delta_0}$ estimates, and a Whitney decomposition.
\end{abstract}
\maketitle

\section{Introduction and main results}\label{sec:intro}
{In this paper, we study the Dirichlet problem}
\begin{equation}\label{eq:original}
\begin{cases}
|Du|^\gamma F(D^2u)=f &\text{in }\Omega,\\
u=g &\text{on }T,
\end{cases}
\end{equation}
{where $-1<\gamma<0$, $\Omega\subset\R^n$ is a domain,
and $T\subset\partial\Omega$ is a relatively open boundary portion.
We assume that $f$ is continuous and bounded.
Throughout the paper, $F:\Sn\to\R$ satisfies $F(0)=0$
and is uniformly elliptic: there are constants $0<\lambda\le\Lambda$ such that
\[
\lambda\|N\|\le F(M+N)-F(M)\le\Lambda\|N\|,
\qquad M,N\in\Sn,\quad N\ge0.
\]
Here $\Sn$ is the space of real symmetric $n\times n$ matrices,
and $\|\cdot\|$ is the operator norm.
The factor $|Du|^\gamma$ is singular at points where $Du=0$.
Our aim is to prove boundary gradient and second derivative estimates
under $C^{1,\alpha}$ assumptions on the boundary and the Dirichlet data.}

{The uniformly elliptic case $\gamma=0$ provides the foundation for our analysis.
We refer to Caffarelli and Cabr\'e~\cite{CC95} for the basic regularity theory.
Silvestre and Sirakov~\cite{SS14} proved boundary gradient
and second derivative estimates on smooth domains.
Lian and Zhang~\cite{LZ20} used compactness and perturbation arguments
to prove pointwise boundary $C^{1,\alpha}$ and $C^{2,\alpha}$ regularity
on $C^{1,\alpha}$ and $C^{2,\alpha}$ domains, respectively.
Their method works directly with the boundary geometry.
For second derivative estimates on $C^{1,\alpha}$ domains,
Li and Li~\cite{LiLi23} proved $W^{2,\delta}$ estimates
for fully nonlinear elliptic inequalities.
Li, Li and Zhang~\cite{li2023} obtained boundary $W^{2,p}$ estimates
on such domains under suitable assumptions.
Their proof combines boundary estimates, interior estimates,
and a Whitney decomposition.}

{For $-1<\gamma<0$, Birindelli and Demengel~\cite{BD04}
introduced a viscosity framework and proved comparison principles
for singular operators.
They later proved boundary $C^{1,\beta}$ regularity for zero Dirichlet
problems on bounded $C^2$ domains~\cite[Corollary~3.3]{BD10}.
Their results require suitable homogeneity and structural assumptions
on the operator.
Their proof combines regularization and a fixed point argument.
It uses uniform Lipschitz bounds and standard boundary estimates
for auxiliary uniformly elliptic equations.
The boundary bounds rely on distance barriers from their earlier work~\cite{BD07}.
For a $C^2$ domain, the distance function $d(x)=\dist(x,\partial\Omega)$
is $C^{1,1}$ in a boundary strip and satisfies $D^2d\le CI$
there with a uniform constant $C$.
This bound controls the Hessian terms in those barriers.
A $C^{1,\alpha}$ boundary alone does not provide this bound.
Thus the same barrier argument cannot be used unchanged under our
pointwise boundary assumption.
We use compactness and perturbation to work directly with the boundary
geometry, without differentiating or flattening the boundary.}

{For second derivatives, Li and Li~\cite{LiLi17}
proved global $W^{2,\delta_0}$ estimates on balls for singular Pucci inequalities
with bounded continuous right-hand sides.
Their theorem is stated on $B_1$, with no smoothness assumption on the
boundary values.
Their proof uses sliding paraboloids and a covering argument adapted to the ball.
In particular, the convexity of $B_1$ keeps the paraboloid vertices inside
$B_1$ in the proof of their Lemma~3.2.
Their Remark~1.1 extends the estimate to domains that are unions of balls
contained in the domain, with uniformly positive radii and bounded overlap.
A general $C^{1,\alpha}$ domain need not be convex or admit such a covering.
It may even have a boundary point with no interior tangent ball.
Thus their global argument does not directly cover all $C^{1,\alpha}$ domains.
We instead combine pointwise boundary $C^{1,\alpha}$ estimates,
interior $W^{2,\delta_0}$ estimates, and a Whitney decomposition.
The boundary estimate gives the decay needed to sum the interior estimates
as the cubes approach the boundary.
Li and Li later obtained estimates with $L^n$ right-hand sides~\cite{LiLi18}.}

{We prove two results for the singular equation~\eqref{eq:original}.
The first gives pointwise boundary $C^{1,\alpha}$ regularity
for every $0<\alpha<1$.
It requires only pointwise $C^{1,\alpha}$ assumptions on the boundary
and the Dirichlet data at the chosen point.
Thus the boundary need not be $C^2$ near that point.
The second gives $W^{2,\delta}$ estimates up to a $C^{1,\alpha}$ boundary
portion with $C^{1,\alpha}$ Dirichlet data.
For this equation, it extends the estimates on balls in~\cite{LiLi17}
to $C^{1,\alpha}$ boundary portions.
The allowed range is $0<\delta\le\delta_0<1$, where $\delta_0$
is the interior exponent.}

{We now state the main results.
The first theorem concerns pointwise regularity at the origin.}
\begin{theorem}[Boundary pointwise $C^{1,\alpha}$ regularity]\label{thm:main}
Let $-1<\gamma<0$ and $0<\alpha<1$. Assume
$F:\Sn\to\R$ {is} uniformly
elliptic with constants $\lambda,\Lambda$ and $F(0)=0$.  Suppose that $\partial\Omega\in C^{1,\alpha}(0)$, i.e., $0\in \partial\Omega$ and there exists $K_\Om\ge0$ such that
\begin{equation}\label{eq:geometry}
B_1\cap\{x_n>K_\Om|x'|^{1+\alpha}\}\subset\Om,\qquad
B_1\cap\{x_n<-K_\Om|x'|^{1+\alpha}\}\subset\Om^c.
\end{equation}
Let $u\in C(\overline\Om\cap B_1)\cap L^\infty(\Om\cap B_1)$ be a singular
viscosity solution, in the sense of Definition~\ref{def:viscosity}, of
\begin{equation}\label{eq:pointwise-problem}
\begin{cases}
|Du|^\gamma F(D^2u)=f(x)&\ \ \mathrm{in }\ \Om\cap B_1,\\
u=g&\ \ \mathrm{on }\ \partial\Om\cap B_1,
\end{cases}
\end{equation}
where  
\begin{equation}\label{eq:pointwise-data}
{f\in C(\Om\cap B_1)\cap L^\infty(\Om\cap B_1)},
\qquad g\in C^{1,\alpha}(0)
\end{equation}
 in the following sense that there are $b\in\R^n$ and $K_g\ge0$ satisfying
\begin{equation}\label{eq:g-data}
|g(x)-g(0)-b\cdot x|\le K_g|x|^{1+\alpha},\qquad \forall x\in\partial\Om\cap B_1.
\end{equation}
Then $u\in C^{1,\alpha}(0)$, i.e.,
there exists  an affine function
$L$ with $L(0)=g(0)$ such that
\begin{equation}\label{eq:main-error}
|u(x)-L(x)| \le C |x|^{1+\alpha}\left(\norm u_{L^\infty(\Om\cap B_1)}+|g(0)|+|b|+K_g
 +\norm f_{L^\infty(\Om\cap B_1)}^{1/(1+\gamma)}\right),\ \ \forall x\in\Om_{r_0},
\end{equation} 
and
\begin{equation}\label{eq:main-slope}
|DL| \le C\left(\norm u_{L^\infty(\Om\cap B_1)}+|g(0)|+|b|+K_g
 +\norm f_{L^\infty(\Om\cap B_1)}^{1/(1+\gamma)}\right),
\end{equation} 
where $r_0$ and $C$ depend on $n,\lambda,\Lambda,\gamma,\alpha$ and $K_{\Om}$.
In addition, the affine expansion is unique.
\end{theorem}

{The second theorem gives a boundary $W^{2,\delta}$ estimate.}

\begin{theorem}[Boundary $W^{2,\delta}$ estimate]
\label{thm:boundary-sobolev}
Let $-1<\gamma<0$ and $0<\alpha<1$. Let $\Omega\subset\R^n$ be a
bounded domain with a relatively open $C^{1,\alpha}$ boundary portion
$T\subset\partial\Omega$.
Assume that $F:\Sn\to\R$ is uniformly elliptic with constants
$\lambda,\Lambda$ and $F(0)=0$.
Suppose that $u\in C(\Omega\cup T)\cap L^\infty(\Omega)$ is a singular
viscosity solution of
\begin{equation}\label{eq:boundary-sobolev-problem}
\begin{cases}
|Du|^\gamma F(D^2u)=f(x)&\text{in }\Omega,\\
u=g&\text{on }T,
\end{cases}
\end{equation}
where
\begin{equation}\label{eq:boundary-sobolev-data}
f\in C(\Omega)\cap L^\infty(\Omega),\qquad
 g\in C^{1,\alpha}(T).
\end{equation}
Then there exists $\delta_0=\delta_0(n,\lambda,\Lambda)\in(0,1)$ such that,
for every $0<\delta\le\delta_0$ and every domain $\Omega'\subseteq\Omega$
with $\overline{\Omega'}\subset\Omega\cup T$, we have
\begin{equation}\label{eq:boundary-sobolev}
\norm u_{W^{2,\delta}(\Omega')}
\le C\left(\norm u_{L^\infty(\Omega)}
+\norm g_{C^{1,\alpha}(T)}
+\norm f_{L^\infty(\Omega)}^{1/(1+\gamma)}\right).
\end{equation}
Here $C$ depends on $n,\lambda,\Lambda,\gamma,\alpha,\delta,T,\Omega'$ and
$\Omega$.
In particular, when $T=\partial\Omega$, the estimate holds with
$\Omega'=\Omega$.
\end{theorem}

{The singular equation requires a different treatment from the degenerate equation.
For $\gamma>0$, the gradient factor vanishes at zero.
After affine subtraction and rescaling, the gradient shift may become large.
The degenerate proof therefore needs estimates that are uniform in this shift.
For $-1<\gamma<0$, the gradient factor is undefined at zero.
We use the zero-gradient result of Birindelli and Demengel~\cite{BD10,BD15}
to pass to a continuous form of the equation.
Its right-hand side grows sublinearly in the gradient.
Bounds on the affine slopes and rescaling give
$|H_k(x,p)|\le\varepsilon_0(1+|p|)$ at every step.
This bound allows the compactness argument to continue throughout the iteration.
Boundary barriers then allow us to adapt the geometric method
of Lian and Zhang~\cite{LZ20} without flattening the boundary.
Birindelli and Demengel~\cite{BD10} treated zero Dirichlet problems on bounded $C^2$ domains.
For the equation considered here, Theorem~\ref{thm:main} allows
pointwise $C^{1,\alpha}$ boundaries and nonzero $C^{1,\alpha}$ data.
There is also a difference in second derivative estimates.
In the degenerate case, boundedness of $u$ and $f$ alone does not give
uniform interior $W^{2,\delta}$ estimates for any $\delta>0$.
In the singular case, we can use the estimates on balls in~\cite{LiLi17}.
The difficulty is to preserve the singular equation and obtain bounds
that can be summed near the boundary.
We subtract a boundary value on each Whitney cube to do this.
Theorem~\ref{thm:boundary-sobolev} thus extends those estimates
to $C^{1,\alpha}$ boundary portions.}

{The proofs proceed as follows.
We first recall that every singular viscosity solution satisfies
$F(D^2u)=f(x)|Du|^{-\gamma}$ in the usual viscosity sense~\cite{BD15}.
After normalization, we construct barriers near an almost flat boundary.
These barriers and interior H\"older estimates give the compactness needed
to pass to a flat boundary problem.
The limit solves a homogeneous uniformly elliptic equation with zero boundary data.
Its boundary regularity gives an affine approximation with error $Cr^2$.
Rescaling and iteration then give an error of $Cr^{1+\alpha}$,
which proves Theorem~\ref{thm:main}.
To prove Theorem~\ref{thm:boundary-sobolev}, we use a Whitney decomposition
as in~\cite{LiLi23,li2023}.
On each cube, we subtract $g(y)$ at a nearby boundary point $y$.
The boundary estimate bounds $|u-g(y)|$ by a constant times the cube diameter $d$.
The interior estimate of~\cite{LiLi17} then bounds
the integral of $|D^2u|^\delta$ by a constant times $d^{n-\delta}$.
Since $\delta<1$, we can sum these bounds over the Whitney cubes.
We estimate the gradient term in the same way.
A finite covering gives the estimate on $\Omega'$.}

{The paper is organized as follows.
Section~\ref{sec:prelim} gives the definitions and basic estimates.
It also recalls the passage from the singular equation to a continuous equation.
Section~\ref{sec:compact} constructs the boundary barriers
and proves Theorem~\ref{thm:main} by compactness and iteration.
Section~\ref{sec:boundary-sobolev} uses Whitney decomposition
to prove the boundary $W^{2,\delta}$ estimates.}

\noindent\textbf{Notation.} \\
1. $e_i=(0,...,0,1,...,0)=i^{th}$ standard coordinate vector.\\
2. $x'=(x_1,x_2,...,x_{n-1})$ and $x=(x',x_n).$\\
3. $\mathbb{R}^n_+=\{x\in \mathbb{R}^n:x_n>0\}.$\\
4. $B_r(x_0)=\{x\in \mathbb{R}^n: |x-x_0|<r\}$ and $B^+_r (x_0) = B_r(x_0) \cap  \mathbb{R}^n_+$.\\
5. $B_r'=\{x'\in \mathbb{R}^{n-1}: |x'|<r\}$.\\
6. $T_r= B_r\cap\{x_n=0\}.$\\
7. $\Omega_r (x_0)= \Omega\cap B_r(x_0)$ and $(\partial\Omega)_r (x_0)= \partial\Omega\cap B_r(x_0)$. We omit $x_0$ when $x_0=0.$\\
8. $\mathbb S(n)$ is the space of real symmetric $n\times n$ matrices with operator norm $\|\cdot\|$.

\section{Basic Definitions and preliminary results}\label{sec:prelim}

\subsection{Pointwise boundary geometry and data}
\begin{definition}[Pointwise $C^{1,\alpha}$ regularity on domains]\label{def:geometry}
Let  $\Omega$ be a bounded domain, $x_0\in\Gamma\subset\partial\Omega$. We say that $\Gamma$ is $C^{1,\alpha}$ at $x_0$ or $\Gamma\in C^{1,\alpha}(x_0)$ if there exist $K,r_0>0$ and a coordinate system $(x_1,x_2,\dots,x_n)$ such that $x_0=0$ in this coordinate system
and
\begin{equation}\label{eq:definition-domain-upper}
B_{r_0}\cap \{(x',x_n)\mid x_n > K|x'|^{1+\alpha}\} \subset B_{r_0}\cap \Omega
\end{equation}
and
\begin{equation}\label{eq:definition-domain-lower}
B_{r_0}\cap \{(x',x_n)\mid x_n < {-}K|x'|^{1+\alpha}\} \subset B_{r_0}\cap \Omega^c.
\end{equation}
Then, define
\[
\|\partial\Omega\|_{C^{1,\alpha}(x_0)} = \inf\big\{K \,\big|\, \text{\eqref{eq:definition-domain-upper} and \eqref{eq:definition-domain-lower} hold for } K\big\}.
\]
\end{definition}

\begin{definition}[Pointwise $C^{1,\alpha}$ regularity on functions]\label{def:function}
{For a bounded set $A\subset\R^n$ and a function $f:A\to\R$,
we write $f\in C^{1,\alpha}(x_0)$ if an affine polynomial $P$
and a constant $K$ satisfy
\begin{equation}\label{eq:definition-function-error}
|f(x)-P(x)|\le K|x-x_0|^{1+\alpha},\qquad x\in A.
\end{equation}
Here $x_0\in A$. For affine $P$, set
\[
\|P\|=|P(x_0)|+|DP|.
\]
Choose $P_0$ among the admissible polynomials so that
\[
\|P_0\|=\min\bigl\{\|P\|:\ \text{there is a }K
\text{ for which \eqref{eq:definition-function-error} holds}\bigr\}.
\]
The derivative and norms are defined by
\[
Df(x_0)=DP_0,\qquad
[f]_{C^{1,\alpha}(x_0)}
=\min\bigl\{K:\ \text{\eqref{eq:definition-function-error} holds with }P_0
\text{ and }K\bigr\},
\]
\[
\|f\|_{C^{1,\alpha}(x_0)}=\|P_0\|+[f]_{C^{1,\alpha}(x_0)}.
\]
This is the $k=1$ case of the definition in~\cite{LZ20}.}
\end{definition}

\begin{remark}
{In Definition~\ref{def:geometry}, we take $r_0=1$.} See 
\cite[Definitions 1.1--1.2]{LZ20} for more details.  
\end{remark}

\subsection{Ellipticity, Pucci operators and viscosity solutions}
\label{subsec:operators}

\begin{definition}[Uniform ellipticity]\label{ue}
An operator $F:\Sn\to\R$ is uniformly elliptic if there are
$0<\lambda\le\Lambda<\infty$ such that, for any $M\in\Sn$,
\begin{equation}\label{eq:ue}
\lambda\norm N\le F(M+N)-F(M)\le\Lambda\norm N,\ \ \forall N\ge0.
\end{equation}
We write $N\ge 0$ whenever $N$ is non-negative definite. Here $\norm N$ is equal to the {maximum eigenvalue} of $N$. 
\end{definition}

\begin{definition}[Pucci operators] 
For $M\in\Sn$ with eigenvalues $\mu_1,\ldots,\mu_n$, define
\[
\begin{aligned}
\Mp(M)&=\Lambda\sum_{\mu_i>0}\mu_i
              +\lambda\sum_{\mu_i<0}\mu_i,\\
\Mm(M)&=\lambda\sum_{\mu_i>0}\mu_i
              +\Lambda\sum_{\mu_i<0}\mu_i.
\end{aligned}
\]
\end{definition}
See \cite[Section~2]{CC95} for the above uniform ellipticity and Pucci conventions.

\begin{remark}
{Throughout} this paper, we assume that $F(0)=0$. By uniform ellipticity
 we have
\[\Mm(M)\le F(M)\le\Mp(M),\qquad
|F(M)-F(N)|\le n\Lambda\norm{M-N},\]
which {will be used frequently} in subsequent sections.
\end{remark}

We introduce the usual test-function definition for continuous equations
\cite[Section~2]{CIL92} and the nonzero-gradient convention for
singular equations \cite[Definition~2.1]{BD10}.

\begin{definition}[Viscosity solutions]\label{def:viscosity}
Let $u\in C(\Omega)$. We define solutions in the following two cases.

\noindent\textbf{(a) Continuous equations.}
{Let $H\in C(\Omega\times\mathbb R^n)$.}
We say that $u$ is a viscosity subsolution (supersolution) of
\[
F(D^2u)=H(x,Du)\qquad\text{in }\Omega
\]
if every $\varphi\in C^2(\Omega)$ such that $u-\varphi$ has a local
maximum (minimum) at $x_0\in\Omega$ satisfies
\[
F(D^2\varphi(x_0))\ge H(x_0,D\varphi(x_0))
\qquad(\text{resp. }\le).
\]
This condition applies to all test gradients, including
$D\varphi(x_0)=0$.

\smallskip\noindent\textbf{(b) Singular equations.}
{Let $-1<\gamma<0$ and $f\in C(\Omega)$.}
We say that $u$ is a singular viscosity subsolution (supersolution) of
\[
|Du|^\gamma F(D^2u)=f(x)\qquad\text{in }\Omega
\]
if, at each $x_0\in\Omega$, the following requirements hold:
\begin{enumerate}
\item[(i)]
If $u$ is not locally constant at $x_0$, every
$\varphi\in C^2(\Omega)$ such that $u-\varphi$ has a local maximum
(minimum) at $x_0$ and $D\varphi(x_0)\ne0$ satisfies
\[
|D\varphi(x_0)|^\gamma F(D^2\varphi(x_0))\ge f(x_0)
\qquad(\text{resp. }\le).
\]
\item[(ii)]
If $u$ is locally constant at $x_0$, there exists
$B_\delta(x_0)\subset\Omega$ on which $u$ is constant and
\[
0\ge f(x)\quad\text{for every }x\in B_\delta(x_0)
\qquad(\text{resp. }\le).
\]
\end{enumerate}
Only nonzero test gradients are used in (i); the locally constant
case is governed by (ii).

For either equation, a viscosity solution is both a subsolution and
a supersolution in the corresponding sense.
\end{definition}

{The following proposition was proved by Birindelli and
Demengel~\cite[Proposition~1.1]{BD15}. It allows us to pass from the
singular equation to a continuous equation.  For completeness, we give an alternative proof
of this known result, following the perturbation argument of
Imbert and Silvestre~\cite[Lemma~6]{IS13}.}

\begin{proposition}[{Continuous form of the singular equation}]
\label{prop:continuous}
{Let $F$ be uniformly elliptic with $F(0)=0$, let
$-1<\gamma<0$, and let $f\in C(\Omega)$.
Every singular viscosity solution of
\[
|Du|^\gamma F(D^2u)=f(x)\qquad\text{in }\Omega
\]
is a viscosity solution of
\[
F(D^2u)=f(x)|Du|^{-\gamma}\qquad\text{in }\Omega
\]
in the usual sense.}
\end{proposition}

\begin{proof}
{Let $\phi\in C^2(\Omega)$ touch $u$ from above at $x_0\in\Omega$.
If $D\phi(x_0)\ne0$, the singular viscosity inequality gives
\[
F(D^2\phi(x_0))\ge f(x_0)|D\phi(x_0)|^{-\gamma}.
\]
If $D\phi(x_0)=0$, it remains to prove $F(D^2\phi(x_0))\ge0$.
We may assume that $x_0=0$ and $u(0)=\phi(0)=0$.
Suppose, to the contrary,
that $D\phi(0)=0$ and $F(D^2\phi(0))<0$. By continuity of $F$, there
are $\tau,c>0$ such that
\begin{equation*}
A=D^2\phi(0)+\tau I,
\qquad F(A)\le-c<0.
\end{equation*}
  By Taylor's formula,
\[\phi(x)=\frac12 x\cdot D^2 \phi(0)x+o(|x|^2).\] 
By {choosing} $r$ small enough, we obtain 
\begin{equation*}
u(x)-\frac12 x\cdot Ax\le {\phi(x)}-\frac12 x\cdot Ax \le-\frac{\tau}{4}|x|^2,\quad x\in B_r\subset\subset\Omega;
\qquad \left(u(x)-\frac12 x\cdot Ax\right)(0)=0.
\end{equation*}
If $A\ge 0$, by ellipticity, we have $F(A)\ge F(0)+\lambda\tr A\ge 0$, which is a contradiction. 
Thus, $A$ has at least one  negative eigenvalue.
Let $E$ be the nontrivial sum of its negative
eigenspaces and let $P$ be the orthogonal projection onto $E$. Then
$AP=PA$, and the restriction of $A$ to $E$ is negative definite.
Choose $C_f\ge\sup_{B_r}|f|$.

For $0<t<\tau r/4$, let
\[\phi_t(x)=u(x)-\frac12 x\cdot Ax+t|Px|,\qquad 
{M_t=\max_{\overline B_r}\phi_t(x)=\phi_t(x_t).}\]
We have
\[\phi_t(x)\le -\tau r^2/4+tr<0\qquad \mathrm{in}\quad \partial B_r.\]
Since $\phi_t(0)=0$, we have $x_t\in B_r$ and $M_t={\phi_t(x_t)}\ge0$.
Moreover, we have
\[0\le-\tau|x_t|^2/4+t|Px_t|\le-\tau|x_t|^2/4+t|x_t|.\]
If $x_t\ne 0$, we divide by $|x_t|$ and obtain $|x_t|\leq 4t/\tau$. Hence,
\[
x_t\longrightarrow0\qquad x_t\in B_r.\]

In the following, we derive a contradiction by considering the two cases $Px_t\ne0$ and  $Px_t=0$.

\noindent \textbf{Case 1. $Px_t\ne0$.}

Set
\[\psi_t(x)=\frac12 x\cdot Ax-t|Px|+M_t.\] 
Then
\[u\le \psi_t,\qquad u(x_t)=\psi_t(x_t).\]
The function $\psi_t$ is $C^2$ near $x_t$, so it is an admissible local test.
Put $y_t=Px_t$. 
Since $y_t\in E$ and $AP=PA$,
\[
y_t\cdot D\psi_t(x_t)=y_t\cdot Ax_t-t|y_t|=y_t\cdot Ay_t-t|y_t|<0,
\]
so 
\[D\psi_t(x_t)\ne0.\]
 Note that the Hessian of $x\mapsto|Px|$ is positive
semidefinite at points with $Px\ne0$. Indeed, for $Px\ne0$,
we have
\[D|Px|=\frac{Px}{|Px|},\qquad D^2 |Px|=\frac{P}{|Px|}-\frac{Px\otimes Px}{|Px|^3}.\]
For any $\xi\in \mathbb{R}^n$, by the {Cauchy--Schwarz} inequality,
\[\xi\cdot D^2|Px|\xi=\frac{|P\xi|^2|Px|^2-(P\xi\cdot Px)^2}{|Px|^3}\ge 0.\]
Therefore
\[D^2 |Px|(x_t)\ge 0,\qquad
D^2\psi_t(x_t)\le A,
\qquad F(D^2\psi_t(x_t))\le F(A)\le-c.
\]
At the same time,
\[|D\psi_t(x_t)|\le\norm A|x_t|+t\to0.\]
 The nonzero-gradient viscosity inequality at $x_t$ gives
\begin{equation*}
-c\ge F(D^2\psi_t(x_t))
\ge-C_f|D\psi_t(x_t)|^{-\gamma}\longrightarrow0,
\end{equation*}
which is a contradiction.

\noindent \textbf{Case 2. $Px_t=0$.}

Fix a unit vector $e\in E$ and set
\[\psi_t(x)=\frac12 x\cdot Ax-t e\cdot Px+M_t.\]
 The inequality
$e\cdot Px\le|Px|$ shows 
\[\psi_t(x)\ge \frac12 x\cdot Ax-t |Px|+M_t\ge u(x).\]
Hence $\psi_t$ is still an upper test at $x_t$.
Now $D^2\psi_t=A$ and
\[
PD\psi_t(x_t)=PAx_t-te=APx_t-te=-te\ne0.
\]
Again the test {gradient is} nonvanishing  but tends to zero, while the viscosity inequality would give
\[
-c\ge-C_f|D\psi_t(x_t)|^{-\gamma}\longrightarrow0,
\] which is a contradiction.

Thus the subsolution inequality holds for every upper test.
Apply the same argument to $-u$, $\widehat F(M)=-F(-M)$, and $-f$
to obtain the supersolution inequality.
Since $(x,p)\mapsto f(x)|p|^{-\gamma}$ is continuous,
$u$ solves the continuous equation in the usual viscosity sense.}
\end{proof}

\begin{remark}\label{rem:one-way}
Only the implication from the singular equation to the continuous equation
is asserted. A constant function solves $F(D^2u)=f(x)|Du|^{-\gamma}$ for every
$f$ when $F(0)=0$, whereas the singular solution definition requires
$f=0$ on its constant region. The reverse implication therefore needs an
additional constant-branch condition.
\end{remark}

\subsection{Two classical regularity results}
We record exactly the external estimates used in the compactness argument.

\begin{proposition}[Interior H\"older compactness]\label{prop:holder}
Let $v$ be continuous and bounded by one in a ball. Suppose, in the
viscosity sense, that
\begin{equation*}
\Mm(D^2v)-|Dv|\le1,
\qquad \Mp(D^2v)+|Dv|\ge-1.
\end{equation*}
On every strictly smaller concentric ball, $v$ has a uniform
$C^\beta$ estimate for some $\beta\in(0,1)$. The exponent and constant
depend only on $n,\lambda,\Lambda$ and the two radii.
\end{proposition}

This is the interior Krylov--Safonov estimate with bounded first-order
terms in the viscosity setting; see \cite{CC95} and the Pucci framework
in \cite{SS14}.  

Finally, since our method treats {a nonsmooth} boundary as a perturbation of a flat boundary, we conclude this section with the following lemma  {on} pointwise boundary $C^{2,\bar\alpha}$ estimates for $F(D^2 u) = 0$ on flat boundaries with zero boundary values. See Lemma 4.1 in \cite{SS14} and Lemma 2.2 in \cite{LZ20}.

\begin{proposition}[Boundary $ C^{2,\bar\alpha}$ regularity for $F(D^2 u) = 0$ on flat domains]
\label{prop:flat}
Let $F$ be uniformly elliptic with $F(0)=0$. Let  $u$ be a viscosity solution of the equation
\begin{equation}\label{u00}
	\begin{cases}
		F(D^2 u) = 0 & \text{in } \ B_1^+, \\
		u = 0 & \text{on }\  T_1.
	\end{cases}
\end{equation}
	Then $u \in C^{2,\bar\alpha}(0)$ and, for some constants $a$ and $b_{in}$ ($1 \le i \le n$),
\[
\big| u(x) - a x_n - {\sum_{i=1}^n b_{in} x_i x_n} \big| \le \bar C |x|^{2+\bar\alpha} \|u\|_{L^\infty(B_1^+)},\quad \forall\, x \in B_{1/2}^+,  
\]
\[
{F\!\left(D^2\!\left(\sum_{i=1}^n b_{in}x_ix_n\right)\right)=0}  
\]
and
\[
|a| + |b_{in}| \le \bar C\|u\|_{L^\infty(B_1^+)},
\]
{where} $\bar\alpha$ and $\bar C$ depend on $n,\lambda$ and $\Lambda$.
\end{proposition}

\begin{remark}\label{re2}
It follows immediately that if $u$ satisfies \eqref{u00}, then
$$ 
\big| u(x) - a x_n \big| \le \bar C |x|^{2} \|u\|_{L^\infty(B_1^+)},\quad \forall\, x \in B_{1/2}^+.$$
\end{remark}

\section{\texorpdfstring{Boundary pointwise $C^{1,\alpha}$ regularity}{Boundary pointwise C1,alpha regularity}}\label{sec:compact}

In the first two subsections, $F$ is uniformly
elliptic with $F(0)=0$, and $H(x,p)$ is continuous.
We consider viscosity solutions of
\begin{equation}\label{eq:small-growth-equation}
F(D^2 u)=H(x,Du).
\end{equation}

\subsection{An explicit barrier in a boundary strip}

\begin{lemma}[Boundary-strip estimate]\label{lem:barrier}
There are universal constants $\varepsilon_b>0$ and $K_b>0$ with the
following property. Let $\Omega$ be open and satisfy 
\begin{equation}\label{eq:slab}
B_1\cap\{x_n>\sigma\}\subset\Omega\cap B_1\subset B_1\cap\{x_n>-\sigma\}
\qquad \mathrm{for}\quad 0\le\sigma\le1/16.
\end{equation}
Let $u\in C(\overline{\Omega}\cap B_1)$ satisfy
\begin{equation}\label{eq:barrier-problem}
F(D^2u)=H(x,Du)\quad\text{in }\Omega\cap B_1,
\end{equation}
and assume that, for some $m\ge0$,
\[
\norm u_{L^\infty(\Omega\cap B_1)}\le1,
\qquad |H(x,p)|\le\varepsilon_b(1+|p|),
\]
\[
|u|\le m\quad\text{on }\partial\Omega\cap B_1.
\]
Then
\begin{equation*}
|u(x)|\le m+K_b(x_n+\sigma),\qquad  x\in\Omega\cap\{|x'|\le1/2,0<x_n+\sigma<1/4\}.
\end{equation*}
\end{lemma}

\begin{proof}
Choose
\[
K_b=4+\frac{8\Lambda(n-1)+1}{\lambda},
\qquad \varepsilon_b=\frac{1}{K_b+9}.
\]
Fix $z'\in\overline B_{1/2}'$ and set
\[
D=\Omega\cap
\{x:|x'-z'|<1/4,\ 0<x_n+\sigma<1/4\}.
\]
Since $|x'|\le3/4$ and $|x_n|\le1/4$ on $\overline D$,
we have $\overline D\subset B_1$. Set
\[
v(x)=m+K_b(x_n+\sigma)
      -2(K_b-4)(x_n+\sigma)^2+16|x'-z'|^2.
\]
In the following, we prove that $v$ provides an upper barrier while  $-v$ provides a lower barrier.

We verify the boundary comparison in four parts. First, on
$\overline D$ we have $0\le x_n+\sigma\le1/4$ and $K_b>4$, so
\[
\begin{aligned}
&K_b(x_n+\sigma)-2(K_b-4)(x_n+\sigma)^2\\
&\qquad=(x_n+\sigma)
       \bigl[K_b-2(K_b-4)(x_n+\sigma)\bigr]\\
&\qquad\ge\left(\frac{K_b}{2}+2\right)(x_n+\sigma)\ge0.
\end{aligned}
\]
Continuity also gives $|u|\le1$ on $\overline D$.
\begin{enumerate}
\item \textit{Lateral boundary.}
If $x\in\partial D$ and $|x'-z'|=1/4$, we have
$16|x'-z'|^2=1$ and then
\[
v(x)\ge m+1\ge1\ge|u(x)|\qquad\mathrm{on}\quad \partial D\cap\{|x'-z'|=1/4\}.
\]

\item \textit{Top boundary.}
If $x\in\partial D$ and $x_n+\sigma=1/4$, then
\[
v(x)=m+\frac{K_b}{8}+\frac12+16|x'-z'|^2
      \ge1\ge|u(x)| \qquad\mathrm{on}\quad \partial D\cap\{x_n+\sigma=1/4\},
\]
where we used $K_b>4$ and $m\ge0$.

\item \textit{The original domain boundary.}
If $x\in\partial D\cap\partial\Omega$, then   $|u(x)|\le m$.
Hence
\[
v(x)\ge m\ge|u(x)|\qquad\mathrm{on}\quad \partial D\cap\partial\Omega.
\]

\item \textit{Bottom boundary.}
If $x\in\partial D$ and  $x_n+\sigma=0$,
since $x\in\overline D\subset\overline{\Omega}\cap B_1$
and $\Omega  \cap B_1\subset\{x_n>-\sigma\}$,
we have $x\notin\Omega$ and therefore $x\in\partial\Omega$.
Hence \[
v(x)=m+16|x'-z'|^2\ge m\ge|u(x)|\qquad\mathrm{on}\quad \partial D\cap\{x_n+\sigma=0\}.
\]
\end{enumerate}
Consequently,
\[
-v\le u\le v\qquad\mathrm{on}\quad\partial D.
\]

Direct calculation gives
\[Dv=(32(x'-z'),K_b-4(K_b-4)(x_n+\sigma)),\qquad
D^2v=\diag(32,\ldots,32,-4(K_b-4)).
\]
Since $|D_{x'}v|\le8$ and
$4\le\partial_nv\le K_b$ in $D$, we obtain
\[|Dv|\le K_b+8\qquad \mathrm{in}\quad D\]
and then
\[|H(x,\pm Dv)|\leq \varepsilon_b(1+|Dv|)\leq \varepsilon_b(9+K_b)=1.\]
By ellipticity,
\begin{align*}
F(D^2v)\le\Mp(D^2v)
=32\Lambda(n-1)-4\lambda(K_b-4)
&=-4< H(x,Dv),\\
  F(-D^2v)\ge\Mm(-D^2v)=-\Mp(D^2v)
=4&>  H(x,-Dv).
\end{align*}
Hence $v$ and $-v$ are smooth strict supersolutions and
subsolutions, respectively. The comparison principle with
smooth strict barriers yields
\[
-v\le u\le v\qquad\text{in }D.
\]
For each point under consideration, choose $z'=x'$.
The term $16|x'-z'|^2$ vanishes, and therefore
\[
|u(x)|\le v(x)\le m+K_b(x_n+\sigma).\qedhere
\]
\end{proof}

\begin{remark}
The above lemma presents a uniform estimate for solutions, which is a kind of "equicontinuity" up to the boundary.
\end{remark}

\subsection{Compactness on varying domains}

\begin{lemma}[One-step affine approximation]\label{lem:one-step}
Fix $0<\alpha<1$ and set $C_0=\max\{1,4\bar C\}$, where $\bar C$ is the universal constant in Proposition~\ref{prop:flat}. There exist
\[
\eta\in(0,1/8),\quad \eta^\alpha\le1/2,\qquad
\varepsilon_0\in(0,\min\{1/16,\varepsilon_b\}],
\]
depending only on $n,\lambda,\Lambda,\alpha$, such that the following
holds. Assume \eqref{eq:slab} with $0\le\sigma\le\varepsilon_0$.
Let $u\in C(\overline{\Omega}\cap B_1)$ satisfy
\begin{equation}\label{eq:one-step-problem}
F(D^2u)=H(x,Du)\quad\text{in }\Omega\cap B_1,
\end{equation}
and assume that
\[
\norm u_{L^\infty(\Omega\cap B_1)}\le1,
\qquad |H(x,p)|\le\varepsilon_0(1+|p|),
\]
\[
|u|\le\varepsilon_0\quad\text{on }\partial\Omega\cap B_1.
\]
Then there exists $c\in\R$ such that
\begin{equation}\label{eq:one-step-result}
|c|\le C_0,
\qquad
\norm{u-cx_n}_{L^\infty(\Omega\cap B_\eta)}\le\eta^{1+\alpha}.
\end{equation}
\end{lemma}

\begin{proof}
We argue by contradiction. Suppose that the conclusion fails. Then
there exist sequences
$\Omega_j,\sigma_j,F_j,H_j,u_j$ such that
\begin{equation}\label{eq:compactness-problems}
F_j(D^2u_j)=H_j(x,Du_j)\quad\text{in }\Omega_j\cap B_1
\end{equation}
with
\[0\le \sigma_j\le \varepsilon_j\searrow 0,\qquad \varepsilon_j\le \min\{1/16,\varepsilon_b\},\qquad B_1\cap\{x_n>\sigma_j\}\subset\Omega_j\cap B_1\subset B_1\cap\{x_n>-\sigma_j\} \]
and
\[
\norm {u_j}_{L^\infty(\Omega_j\cap B_1)}\le1,
\qquad |H_j(x,p)|\le\varepsilon_j(1+|p|),
\]
\[
|u_j|\le\varepsilon_j\quad\text{on }\partial\Omega_j\cap B_1,
\]
but
\begin{equation}\label{eq:but}
\|u_j-cx_n\|_{L^\infty(\Omega_j\cap B_\eta)}
>\eta^{1+\alpha}\qquad\text{for every }|c|\le  C_0,
\end{equation}
where $\eta\in(0,1/8)$ is taken small such that
\[
C_0\eta^2\le\displaystyle{\frac1{4}}\eta^{1+\alpha},
\qquad \eta^\alpha\le\displaystyle{\frac1{2}}.
\]

In the following, we proceed in four steps to obtain a contradiction.

\smallskip\noindent\textbf{Step 1: Interior compactness.}
By uniform ellipticity and $F_j(0)=0$,
\[
|F_j(M)-F_j(N)|\le n\Lambda\|M-N\|,
\qquad |F_j(M)|\le n\Lambda\|M\|,\quad \mathrm{for}\ M,N\in\Sn.\]
By Arzel\`a--Ascoli, after passing to a subsequence (still denoted by $F_j$), we obtain
\[
F_j\longrightarrow F_\infty\quad\text{locally uniformly on }\Sn,
\qquad F_\infty(0)=0 
\]
and 
\[
\Mm(M-N)\le F_\infty(M)-F_\infty(N)\le\Mp(M-N).
\]
Since  $u_j$ satisfies $F_j(D^2 u_j)=H_j(x,Du_j)$ with $|H_j(x,p)|\le\varepsilon_j(1+|p|)\le 1+|p|$,
\[
\Mp(D^2u_j)+|Du_j|\ge-1,
\qquad \Mm(D^2u_j)-|Du_j|\le1
\]
in the viscosity sense. For every $K\Subset B_1^+$, since $B_1\cap\{x_n>\sigma_j\}\subset\Omega_j\cap B_1$ with $\sigma_j\rightarrow 0$,
\[
\dist\bigl(K,\partial(\Omega_j\cap B_1)\bigr)
\ge\tfrac12\dist(K,\partial B_1^+)>0
\qquad \text{for\ large}\ j.
\]
By Proposition~\ref{prop:holder}, there exists a subsequence (still denoted by $u_j$) such that
\[u_j\longrightarrow u_\infty\quad\text{locally uniformly in }B_1^+,
\qquad |u_\infty|\le1.
\]

\smallskip\noindent\textbf{Step 2: The limit equation.}
Let $\phi\in C^2$ touch $u_\infty$ from above at $x_0\in B_1^+$.
Choose $r>0$ such that
\[
\overline B_r(x_0)\Subset B_1^+,
\qquad
(u_\infty-\phi)(x_0)=0,
\qquad u_\infty-\phi\le0\quad\text{on }\overline B_r(x_0).
\]
Set $\psi(x)=\phi(x)+|x-x_0|^4$. Then
\[
u_\infty(x)-\psi(x)\le-|x-x_0|^4,
\qquad
D\psi(x_0)=D\phi(x_0),\quad D^2\psi(x_0)=D^2\phi(x_0).
\]
For large $j$, $\overline B_r(x_0)\subset\Omega_j\cap B_1$.
Let $x_j$ be the maximum point of $u_j-\psi$ in $\overline B_r(x_0)$, that is,
\[
(u_j-\psi)(x_j)=\max_{\overline B_r(x_0)}(u_j-\psi).
\]
Hence
\[
|x_j-x_0|^4
\le(u_\infty-\psi)(x_0)-(u_\infty-\psi)(x_j) 
\le2\|u_j-u_\infty\|_{L^\infty(B_r(x_0))}
\rightarrow0.
\]
Thus $x_j\to x_0$ and $x_j\in B_r(x_0)$ for large $j$.
By Definition~\ref{def:viscosity}(a),
\[
F_j(D^2\psi(x_j))
\ge H_j(x_j,D\psi(x_j)) 
\ge-\varepsilon_j\bigl(1+\|D\psi\|_{L^\infty(B_r(x_0))}\bigr)\rightarrow0.
\]
Moreover, we obtain
\[
\bigl|F_j(D^2\psi(x_j))-F_\infty(D^2\phi(x_0))\bigr|
 \le n\Lambda\|D^2\psi(x_j)-D^2\phi(x_0)\|
 +\bigl|(F_j-F_\infty)(D^2\phi(x_0))\bigr|
 \rightarrow0,
\]
which implies
\[
F_\infty(D^2\phi(x_0))
=\lim_{j\to\infty}F_j(D^2\psi(x_j))\ge0.
\]
For a lower test, use $\psi(x)=\phi(x)-|x-x_0|^4$ and minima. The same argument gives
\[
F_j(D^2\psi(x_j))\le\varepsilon_j
\bigl(1+\|D\psi\|_{L^\infty(B_r(x_0))}\bigr)\longrightarrow0,
\qquad F_\infty(D^2\phi(x_0))\le0.
\]
Consequently, 
\[
F_\infty(D^2u_\infty)=0\qquad\text{in }B_1^+
\]
in the viscosity sense.

\smallskip\noindent\textbf{Step 3: Zero boundary values and flat approximation.}
Fix $x$ with $|x'|<1/2$ and $0<x_n<1/4$. For large $j$,
\[
\sigma_j<\min\{x_n,1/4-x_n\}
\quad\Longrightarrow\quad
x\in\Omega_j\cap B_1,\qquad 0<x_n+\sigma_j<1/4.
\]
Applying Lemma~\ref{lem:barrier}  with $m=\varepsilon_j$, we obtain
\[
|u_\infty(x)|
=\lim_{j\to\infty}|u_j(x)|
\le\lim_{j\to\infty}\bigl[\varepsilon_j+K_b(x_n+\sigma_j)\bigr]
=K_bx_n\ \ \mathrm{for}\ \ x\in\{|x'|<1/2,0<x_n<1/4\}.
\]
Fix $z=(z',0)\in T_{1/2}$. We choose 
$0<r<\min\{1/4,1/2-|z'|\}$ such that
\[B_1^+\cap B_r(z)\subset\{|x'|<1/2,0<x_n<1/4\}.\]
Hence
\[
\sup_{x\in B_r(z)\cap B_1^+}|u_\infty(x)|\le K_br
\rightarrow0,\  r\rightarrow0.
\]
Consequently, the limit function $u_\infty\in C(B_1^+\cup T_{1/2})$ satisfies
\begin{equation}\label{eq:limit-boundary-problem}
\begin{cases}
F_\infty(D^2u_\infty)=0&\text{in }B_{1/2}^+,\\
u_\infty=0&\text{on }T_{1/2},
\end{cases}
\end{equation}
with
\[\|u_\infty\|_{L^\infty(B_{1/2}^+)}\le1.\]
Apply Proposition~\ref{prop:flat} to $y\mapsto u_\infty(y/2)$, whose normalized operator is $X\mapsto\frac14F_\infty(4X)$. Since $C_0\ge4\bar C$, there exists
$a\in\R$ such that
\[
|a|\le C_0,
\qquad |u_\infty(x)-ax_n|\le C_0|x|^2
\quad \mathrm{for\ any}\ x\in B_{1/4}^+.
\]

\smallskip\noindent\textbf{Step 4: Transfer to the varying domains.}
 Choose
\[
0<s<\min\{\eta,1/8,\eta^{1+\alpha}/[8(K_b+C_0)]\}.
\]
For large $j$, we have $s+\sigma_j<1/4$ and
\begin{equation*}
\varepsilon_j+(K_b+C_0)\sigma_j\le \eta^{1+\alpha}/8.
\end{equation*}
If $x\in\Omega_j\cap B_\eta$ and $x_n\le s$, then
$-\sigma_j<x_n\le s$. By Lemma \ref{lem:barrier},
\begin{align*}
|u_j(x)-ax_n|
\le\varepsilon_j+K_b(x_n+\sigma_j)+C_0|x_n|
\le\varepsilon_j+(K_b+C_0)(s+\sigma_j)
\le\eta^{1+\alpha}/4.
\end{align*}
For $x\in\Omega_j\cap B_\eta$ with $x_n\ge s$, we have $
\overline B_\eta\cap\{x_n\ge s\}\Subset B_{1/4}^+$ and then
$\sup_{\overline B_\eta\cap\{x_n\ge s\}}|u_j-u_\infty|
\le \eta^{1+\alpha}/4$ for large $j$. 
It then follows that
\[
|u_j(x)-ax_n|
\le |u_j(x)-u_\infty(x)|+C_0\eta^2
\le \eta^{1+\alpha}/2.
\]
Combining the two regions, we obtain
\[
\|u_j-ax_n\|_{L^\infty(\Omega_j\cap B_\eta)}
<\eta^{1+\alpha},\qquad |a|\le C_0,
\]
contrary to \eqref{eq:but} with  $c=a$.
\end{proof}

\subsection{Proof of the boundary pointwise estimate}\label{sec:mainproof}

We organize the proof as in \cite[proof of Theorem~1.6]{LZ20}:
normalize the data, construct affine approximations by induction, and
pass to their limit.

\begin{proof}[Proof of Theorem~\ref{thm:main}]
We proceed in three steps.

\par\medskip\noindent\textbf{Step 1. Normalization of the data.}\par\nobreak
By {Proposition~\ref{prop:continuous}}, $u$ satisfies
\begin{equation}\label{eq:continuous-boundary-problem}
\begin{cases}
F(D^2u)=f(x)|Du|^{-\gamma}&\text{in }\Omega_1,\\
u=g&\text{on }\partial\Omega\cap B_1,
\end{cases}
\end{equation}
in the ordinary viscosity sense. {Set
\[
\mathcal N=\|u\|_{L^\infty(\Omega_1)}+|g(0)|+|b|+K_g
+\|f\|_{L^\infty(\Omega_1)}^{1/(1+\gamma)}.
\]
}If $\mathcal N=0$, then $u=0$ and
we take $L=0$. Assume $\mathcal N>0$.
Fix $C_0,\eta,\varepsilon_0$ as in Lemma~\ref{lem:one-step}, and set
\[
\varepsilon_* = \frac{\varepsilon_0}{2+2C_0},
\qquad M=\varepsilon_*^{-1/(1+\gamma)}\mathcal N.
\]
Choose
\begin{equation}\label{eq:rho}
\rho=\min\left\{\frac14,
 \left(\frac{\varepsilon_*}{1+K_\Omega}\right)^{1/\alpha}\right\}.
\end{equation}
For $y\in\rho^{-1}\Omega\cap B_1$, define
\begin{equation}\label{eq:full-normalization}
\begin{aligned}
w(y)&=\frac{u(\rho y)-g(0)-\rho b\cdot y}{M},\\
F_\rho(X)&=\frac{\rho^2}{M}
 F\!\left(\frac{M}{\rho^2}X\right),
\qquad
f_\rho(y)=\frac{\rho^{2+\gamma}}{M^{1+\gamma}}f(\rho y).
\end{aligned}
\end{equation}
For a test function $\phi$ of $w$, the corresponding test function of
$u$ is $g(0)+b\cdot x+M\phi(x/\rho)$, with gradient and Hessian
\[
b+\frac{M}{\rho}D\phi
 =\frac{M}{\rho}\left(\frac{\rho b}{M}+D\phi\right),
\qquad \frac{M}{\rho^2}D^2\phi,
\]
respectively. Thus
\begin{equation}\label{eq:normalized-boundary-problem}
\begin{cases}
F_\rho(D^2w)=f_\rho(y)
 \left|\dfrac{\rho b}{M}+Dw\right|^{-\gamma}
 &\text{in }\rho^{-1}\Omega\cap B_1,\\[5pt]
w(y)=\dfrac{g(\rho y)-g(0)-\rho b\cdot y}{M}
 &\text{on }\partial(\rho^{-1}\Omega)\cap B_1.
\end{cases}
\end{equation}
Here $F_\rho(0)=0$, and $F_\rho$ has the same ellipticity constants as
$F$. Since $0<1+\gamma<1$ and $M\ge\mathcal N$,
\[
\begin{aligned}
\|w\|_{L^\infty(\rho^{-1}\Omega\cap B_1)}
&\le\frac{\|u\|_\infty+|g(0)|+\rho|b|}{M}\le1,\\
\|f_\rho\|_{L^\infty(\rho^{-1}\Omega\cap B_1)}
&\le\varepsilon_*\rho^{2+\gamma}
 \frac{\|f\|_\infty}{\mathcal N^{1+\gamma}}
 \le\varepsilon_*,\\
\left|\frac{\rho b}{M}\right|&\le1,
\qquad K_\Omega\rho^\alpha\le\varepsilon_*.
\end{aligned}
\]
Moreover, for $y\in\partial(\rho^{-1}\Omega)\cap B_1$,
\[
|w(y)|\le\frac{K_g\rho^{1+\alpha}}{M}|y|^{1+\alpha}
\le\varepsilon_*^{1/(1+\gamma)}\rho^{1+\alpha}|y|^{1+\alpha}
\le\varepsilon_*|y|^{1+\alpha}.
\]
The scaled geometric inclusions are
\[
\begin{aligned}
B_1\cap\{y_n>\varepsilon_*|y'|^{1+\alpha}\}
&\subset\rho^{-1}\Omega,\\
B_1\cap\{y_n<-\varepsilon_*|y'|^{1+\alpha}\}
&\subset(\rho^{-1}\Omega)^c.
\end{aligned}
\]

\par\medskip\noindent\textbf{Step 2. Induction.}\par\nobreak
We construct $a_k\in\R$, with $a_0=0$, such that
\begin{equation}\label{ak}
\begin{aligned}
\|w-a_ky_n\|_{L^\infty(\rho^{-1}\Omega\cap B_{\eta^k})}
&\le\eta^{k(1+\alpha)} &&(k\ge0),\\
|a_k-a_{k-1}|&\le C_0\eta^{(k-1)\alpha} &&(k\ge1).
\end{aligned}
\end{equation}
The case $k=0$ follows from $\|w\|_\infty\le1$.
Assume that \eqref{ak} holds through $k$. Since $\eta^\alpha\le1/2$,
\[
|a_k|\le C_0\sum_{j=0}^{k-1}\eta^{j\alpha}
\le\frac{C_0}{1-\eta^\alpha}\le2C_0,
\]
where the sum is zero when $k=0$.
Set $r=\eta^k$ and
\[
\begin{aligned}
v_k(y)&=\frac{w(ry)-a_k r y_n}{r^{1+\alpha}},\\
F_k(X)&=r^{1-\alpha}F_\rho(r^{\alpha-1}X),\\
H_k(y,p)&=r^{1-\alpha}f_\rho(ry)
 \left|\frac{\rho b}{M}+a_ke_n+r^\alpha p\right|^{-\gamma}.
\end{aligned}
\]
Then
\begin{equation}\label{eq:iteration-pde}
\begin{cases}
F_k(D^2v_k)=H_k(y,Dv_k)
 &\text{in }(\rho r)^{-1}\Omega\cap B_1,\\[4pt]
v_k(y)=\dfrac{w(ry)-a_k r y_n}{r^{1+\alpha}}
 &\text{on }\partial((\rho r)^{-1}\Omega)\cap B_1.
\end{cases}
\end{equation}
We verify the assumptions of Lemma~\ref{lem:one-step}.
First, \eqref{ak} gives $\|v_k\|_{L^\infty((\rho r)^{-1}\Omega\cap B_1)}\le1$,
and $F_k(0)=0$. For $N\ge0$,
\[
\begin{aligned}
\lambda\|N\|
&\le r^{1-\alpha}\bigl[
 F_\rho(r^{\alpha-1}(X+N))-F_\rho(r^{\alpha-1}X)\bigr]\\
&=F_k(X+N)-F_k(X)\le\Lambda\|N\|.
\end{aligned}
\]
Next, $0<-\gamma<1$ implies
\[
|z_1+z_2|^{-\gamma}\le|z_1|^{-\gamma}+|z_2|^{-\gamma},
\qquad |p|^{-\gamma}\le1+|p|.
\]
Consequently,
\begin{equation}\label{eq:iteration-H}
\begin{aligned}
|H_k(y,p)|
&\le\varepsilon_*r^{1-\alpha}
 \bigl[(1+2C_0)^{-\gamma}+r^{-\alpha\gamma}|p|^{-\gamma}\bigr]\\
&\le\varepsilon_*(2+2C_0+|p|)
\le\varepsilon_0(1+|p|).
\end{aligned}
\end{equation}
The geometric inclusions in STEP~1 yield
\[
B_1\cap\{y_n>\varepsilon_*r^\alpha\}
\subset(\rho r)^{-1}\Omega\cap B_1
\subset B_1\cap\{y_n>-\varepsilon_*r^\alpha\},
\qquad \varepsilon_*r^\alpha\le\varepsilon_0.
\]
Finally, for $y\in\partial((\rho r)^{-1}\Omega)\cap B_1$,
\[
|y_n|\le\varepsilon_*r^\alpha|y'|^{1+\alpha},
\]
and therefore
\[
\begin{aligned}
|v_k(y)|
&\le r^{-1-\alpha}|w(ry)|+|a_k|r^{-\alpha}|y_n|\\
&\le\varepsilon_*|y|^{1+\alpha}
 +2C_0\varepsilon_*|y'|^{1+\alpha}
\le(1+2C_0)\varepsilon_*\le\varepsilon_0.
\end{aligned}
\]
Lemma~\ref{lem:one-step} now provides $\bar a\in\R$ such that
\[
|\bar a|\le C_0,
\qquad
\|v_k-\bar a y_n\|_{L^\infty((\rho r)^{-1}\Omega\cap B_\eta)}
\le\eta^{1+\alpha}.
\]
Define $a_{k+1}=a_k+r^\alpha\bar a$. Then
\[
\begin{aligned}
|a_{k+1}-a_k|&\le C_0r^\alpha=C_0\eta^{k\alpha},\\
\|w-a_{k+1}y_n\|_{L^\infty(\rho^{-1}\Omega\cap B_{\eta r})}
&=r^{1+\alpha}
 \|v_k-\bar a y_n\|_{L^\infty((\rho r)^{-1}\Omega\cap B_\eta)}\\
&\le r^{1+\alpha}\eta^{1+\alpha}
 =\eta^{(k+1)(1+\alpha)}.
\end{aligned}
\]
Thus \eqref{ak} holds for $k+1$, completing the induction.

\par\medskip\noindent\textbf{Step 3. Pointwise $C^{1,\alpha}$ regularity.}\par\nobreak
For $m>k$, \eqref{ak} gives
\[
|a_m-a_k|\le C_0\sum_{j=k}^{m-1}\eta^{j\alpha}
\le2C_0\eta^{k\alpha}.
\]
Hence the entire sequence is Cauchy, and its limit $a$ satisfies
\[
a_k\longrightarrow a,
\qquad |a|\le2C_0,
\qquad |a-a_k|\le2C_0\eta^{k\alpha}.
\]
For $y\in\rho^{-1}\Omega$ with $0<|y|<\eta$, choose $k\ge1$ such that
$\eta^{k+1}\le|y|<\eta^k$. Then
\begin{equation}\label{eq:normalized-result}
\begin{aligned}
|w(y)-ay_n|
&\le|w(y)-a_ky_n|+|a_k-a|\,|y_n|\\
&\le\eta^{k(1+\alpha)}+2C_0\eta^{k\alpha}|y|\\
&\le(1+2C_0)\eta^{-(1+\alpha)}|y|^{1+\alpha}.
\end{aligned}
\end{equation}
Returning to the original variables, set
\begin{equation}\label{eq:final-affine}
L(x)=g(0)+b\cdot x+\frac{M}{\rho}ax_n,
\qquad r_0=\rho\eta.
\end{equation}
For $x\in\Omega\cap B_{r_0}$,
\[
|u(x)-L(x)|
=M|w(x/\rho)-a x_n/\rho| 
\le CM\rho^{-(1+\alpha)}|x|^{1+\alpha}
 \le C\mathcal N|x|^{1+\alpha}\]
and
\[|DL|\le|b|+2C_0M/\rho\le C\mathcal N.
\]
Here $C$ and $r_0$ depend only on
$n,\lambda,\Lambda,\gamma,\alpha,K_\Omega$.
This proves \eqref{eq:main-error}--\eqref{eq:main-slope}.

To prove uniqueness, suppose $u-L_i=o(|x|)$ in $\Omega$, $i=1,2$.
Continuity gives $L_i(0)=g(0)$. For every fixed $v$ with $v_n>0$,
\[
tv_n>K_\Omega t^{1+\alpha}|v'|^{1+\alpha}
\quad\text{for all sufficiently small }t>0,
\qquad tv\in\Omega.
\]
Thus
\[
|(DL_1-DL_2)\cdot v|
=\lim_{t\downarrow0}\frac{|L_1(tv)-L_2(tv)|}{t}=0.
\]
The linear functional $(DL_1-DL_2)\cdot v$ vanishes on an open
half-space, so $DL_1=DL_2$ and $L_1=L_2$.
\end{proof}

\section{\texorpdfstring{Boundary $W^{2,\delta}$ estimates on $C^{1,\alpha}$ domains}{Boundary W2,delta estimates on C1,alpha domains}}
\label{sec:boundary-sobolev}
In this section, we  prove Theorem \ref{thm:boundary-sobolev}   by interior $W^{2,\delta_0}$  estimates (Corollary \ref{lem:lili}), boundary  $C^{1,\alpha}$ estimates (Theorem \ref{thm:main}) and  Whitney decomposition (Lemma \ref{l2.1}  and \ref{lf}).

\subsection{Whitney decomposition}
In what follows, by a cube we mean a closed cube in $\mathbb{R}^n$, with sides parallel to the axes. We say two such cubes are disjoint if their interiors are disjoint.

\bigskip

\begin{lemma}[Whitney decomposition]\label{l2.1}
Let $\Omega$ be a non-empty {proper} open set in $\mathbb{R}^n$. Then there exist two sequences of cubes $Q_k$(called the Whitney cubes of $\Omega$) and $\widetilde Q_k=\frac65{Q_k}$ ($\frac 65-$dilation of $Q_k$ with respect to center of $Q_k$) such that

(i) $\Omega=\bigcup_{k=1}^{\infty}Q_{k}=\bigcup_{k=1}^{\infty}\widetilde Q_k$;

(ii) The $Q_k$ are mutually disjoint;

(iii) $d_k\leq \mathrm{dist}\ (Q_k,\partial\Omega)\leq 4d_k$, where $d_k=\diam Q_k$;

(iv) Each point of $\Omega$ is contained in at most $12^n$ of the cubes $\widetilde Q_k$.
\end{lemma}

For the proof of the above lemma, we refer to Theorem 1 and Proposition 3 in Section VI.1 in \cite{ST}.
The following   lemma is proved in \cite{li2023}.

\begin{lemma}\label{lf}
Suppose that $\Omega$ satisfies 
\begin{equation}\label{domain}
0\in\partial\Omega,\quad
\Omega\cap B_1=\{x\in B_1:x_n>\psi(x')\},\quad
\partial\Omega\cap B_1=\{x\in B_1:x_n=\psi(x')\}.
\end{equation}
for some function $\psi \in C^{0,1}(B_1')$.  We have the following conclusions.

(i)\begin{equation}\label{1/4}
\Omega_{1/{12}}\subset\bigcup_{\widetilde Q_k\subset \Omega_{1/{4}}} Q_k.
\end{equation}

(ii)If $q>n-1$, then
\begin{equation}\label{sum}
\sum\limits_{\widetilde{Q}_k\subset \Omega_{1/4}} d_k^q\leq C,
\end{equation}
where $C$ depends on $n$, $q$ and $||\psi||_{C^{0,1}(B_1')}$.
\end{lemma}

\subsection{\texorpdfstring{Interior $W^{2,\delta_0}$ estimates}{Interior W2,delta0 estimates}}

\begin{theorem}[Interior $W^{2,\delta_0}$ estimate for singular inequalities]\label{thm:lili-inequalities}
 Let $-1<\gamma<0$ and suppose
$u\in C(B_1)\cap L^\infty(B_1)$ and $f\in C(B_1)\cap L^\infty(B_1)$
satisfy, in the singular viscosity sense,
\begin{equation}\label{eq:lili-inequalities}
|Du|^\gamma\Mm(D^2u)-|Du|^{1+\gamma}
\le f(x)\le
|Du|^\gamma\Mp(D^2u)+|Du|^{1+\gamma}
\quad\text{in }B_1.
\end{equation}
Then there {exists a constant $\delta_0\in(0,1)$} depending on $n,\lambda,\Lambda$ such that
\begin{equation}\label{eq:lili-interior}
\left(\int_{B_{1/2}}\bigl(|Du|^{\delta_0}+|D^2u|^{\delta_0}\bigr)\,dx
\right)^{1/\delta_0}
\le C\left(\norm u_{L^\infty(B_1)}
+\norm f_{L^\infty(B_1)}^{1/(1+\gamma)}\right),
\end{equation}
where $C$ depends on $n,\lambda,\Lambda,\gamma$ and $\delta_0$.
\end{theorem}

The proof of Theorem \ref{thm:lili-inequalities} can be found in \cite{LiLi17}. By uniform ellipticity, any singular viscosity solution $u$ to the equation $|Du|^\gamma F(D^2u)=f$ automatically satisfies the inequality   \eqref{eq:lili-inequalities}. This observation enables us to apply Theorem \ref{thm:lili-inequalities}  to obtain the corresponding interior estimate for singular equations, as stated in the following corollary.

\begin{corollary}[Interior $W^{2,\delta_0}$ estimate for singular equation]\label{lem:lili}
If $u\in C(B_1)\cap L^\infty(B_1)$ is a singular viscosity solution of
\[
|Du|^\gamma F(D^2 u)=f\quad\text{in }B_1,
\qquad f\in C(B_1)\cap L^\infty(B_1),
\]
then \eqref{eq:lili-interior} holds for the interior exponent
$\delta_0$ fixed in Theorem~\ref{thm:lili-inequalities}.
\end{corollary}

\subsection{\texorpdfstring{Boundary $W^{2,\delta}$ estimates}{Boundary W2,delta estimates}}
Theorem \ref{thm:boundary-sobolev} follows easily from the following Theorem~\ref{thm:local-whitney} which is proved by interior $W^{2,\delta_0}$  estimates (Corollary \ref{lem:lili}), boundary  $C^{1,\alpha}$ estimates (Theorem \ref{thm:main}) and  Whitney decomposition (Lemma \ref{l2.1}  and \ref{lf}).
\begin{theorem}[Local boundary estimate]\label{thm:local-whitney}
Let $0<\alpha<1$, $-1<\gamma<0$, and let $F$ satisfy \eqref{eq:ue}
and $F(0)=0$. Suppose that \eqref{domain} holds with
$\psi\in C^{1,\alpha}(B'_1)$ and that
$u\in C(\overline\Omega\cap B_1)\cap L^\infty(\Omega_1)$ satisfies
\begin{equation}\label{eq:local-whitney-problem}
\begin{cases}
|Du|^\gamma F(D^2u)=f(x)&\text{in }\Omega\cap B_1,\\
u=g&\text{on }\partial\Omega\cap B_1,
\end{cases}
\end{equation}
where $f\in C(\Omega_1)\cap L^\infty(\Omega_1)$ and
$g\in C^{1,\alpha}(\partial\Omega\cap B_1)$.
Let $\delta_0$ be as in Corollary~\ref{lem:lili}.
For every
\begin{equation}\label{eq:boundary-exponents}
0<\delta\le\delta_0,\qquad\delta<1,
\end{equation}
we have
\begin{equation}\label{eq:local-whitney}
\norm u_{W^{2,\delta}(\Omega_{1/12})}
\le C\left(\norm u_{L^\infty(\Omega_1)}
+\norm g_{C^{1,\alpha}(\partial\Omega\cap B_1)}
+\norm f_{L^\infty(\Omega_1)}^{1/(1+\gamma)}\right).
\end{equation}
The constant depends on $n,\lambda,\Lambda,\gamma,\alpha,\delta_0,\delta$
and $\norm\psi_{C^{1,\alpha}(B'_1)}$.
\end{theorem}

\begin{proof}
Write
\[
\mathcal N=\norm u_{L^\infty(\Omega_1)}
+\norm g_{C^{1,\alpha}(\partial\Omega\cap B_1)}
+\norm f_{L^\infty(\Omega_1)}^{1/(1+\gamma)}.
\]

Let $\{Q_k\}_{k=1}^{\infty}$ be Whitney decomposition of $\Omega_{1}$, $d_k=\diam Q_k$ and ${\widetilde Q_k}=\frac 65 Q_k$.
For any $\widetilde Q_k\subset \Omega_{1/4}$, choose
 $y_k\in (\partial\Omega)_{1/2}$ and $\tilde x_k\in \partial\widetilde Q_k$ such that
$$|\tilde x_k-y_k|=\dist(\widetilde Q_k, \partial\Omega_1)<\dist(Q_k, \partial\Omega_1)\leq 4d_k,$$
where Lemma \ref{l2.1} (iii) is used in the last inequality.
Consequently, we see that
$$ |x-y_k|\leq|x-\tilde x_k|+|\tilde x_k-y_k|\leq (\frac65+4)d_k<6d_k,\ \ \forall x\in\widetilde Q_k.$$

By  Theorem \ref{thm:main}, $u$ is $C^{1,\alpha}$ at $y_k$ and then there exists an affine function $L_{y_k}$  such that
\[L_{y_k}(y_k)=g(y_k),\qquad |D L_{y_k}|\le C\mathcal N\]
and
\begin{equation}\label{3.4}
|u(x)-L_{y_k}(x)|\leq C |x-y_k|^{1+\alpha}\mathcal{N}\leq  C d_k^{1+\alpha}\mathcal{N},\ \forall x\in \widetilde Q_k,
\end{equation}
where $C$ depends on $n,\lambda,\Lambda,\alpha,{\gamma}$ and $||\psi||_{C^{1,\alpha}(B_1')}$.
{For $|x-y_k|$ outside the uniform radius in Theorem~\ref{thm:main}, the same bound follows, after enlarging $C$, from the bounds for $u$, $g(y_k)$ and $DL_{y_k}$.}
 Hence  
\begin{equation}\label{eq:whitney-boundary-growth}
|u(x)-g(y_k)|
\le |u(x)-L_{y_k}(x)|
   +|L_{y_k}(x)-g(y_k)|\le C\mathcal N d_k^{1+\alpha}
   +|DL_{y_k}|\,|x-y_k|
\le C\mathcal N d_k.
\end{equation}

Since  $u-g(y_k)$ {still satisfies the singular equation}
\[|D(u-g(y_k))|^\gamma F(D^2(u-g(y_k)))=f\quad\text{in }\ \ \widetilde Q_k,\]
by the scaled interior estimate in Corollary~\ref{lem:lili} and \eqref{eq:whitney-boundary-growth},
\begin{equation}\label{eq:whitney-estimate}
\begin{aligned}
\int_{Q_k}|D^2u|^{\delta_0}\,dx
&\le C\Bigl\{
 d_k^{n-2\delta_0}
 \norm{u-g(y_k)}_{L^\infty(\widetilde Q_k)}^{\delta_0}
 +d_k^{\,n-\frac{\gamma\delta_0}{1+\gamma}}
 \norm f_{L^\infty(\widetilde Q_k)}^{\frac{\delta_0}{1+\gamma}}
 \Bigr\}\\
&\le C\mathcal N^{\delta_0}
 \Bigl(d_k^{n-\delta_0}
       +d_k^{\,n-\frac{\gamma\delta_0}{1+\gamma}}\Bigr)
 \le C\mathcal N^{\delta_0}d_k^{n-\delta_0}.
\end{aligned}
\end{equation}
For $0<\delta\le\delta_0$, by {H\"older's inequality} and the preceding
estimate,
\[
\begin{aligned}
\int_{Q_k}|D^2u|^\delta\,dx
&\le |Q_k|^{1-\delta/\delta_0}
 \left(\int_{Q_k}|D^2u|^{\delta_0}\,dx\right)^{\delta/\delta_0}\\
&\le C\mathcal N^\delta
 d_k^{n(1-\delta/\delta_0)+(n-\delta_0)\delta/\delta_0}
 =C\mathcal N^\delta d_k^{n-\delta}.
\end{aligned}
\]

Summing over the Whitney cubes, we obtain
\[
\sum_{\widetilde Q_k\subset\Omega_{1/4}}
 \int_{Q_k}|D^2u|^\delta\,dx
\le C\mathcal N^\delta
 \sum_{\widetilde Q_k\subset\Omega_{1/4}}d_k^{n-\delta}.
\]
Since $\delta<1$, we have $n-\delta>n-1$. Consequently,
Lemma~\ref{lf}(i)--(ii) yields
\[
\begin{aligned}
\int_{\Omega_{1/12}}|D^2u|^\delta\,dx
&\le\sum_{\widetilde Q_k\subset\Omega_{1/4}}
 \int_{Q_k}|D^2u|^\delta\,dx\\
&\le C\mathcal N^\delta
 \sum_{\widetilde Q_k\subset\Omega_{1/4}}d_k^{n-\delta}
 \le C\mathcal N^\delta.
\end{aligned}
\]
The gradient term follows in the same way from the scaled interior
estimate:
\[
\int_{Q_k}|Du|^{\delta_0}\,dx
\le C\Bigl\{
 d_k^{n-\delta_0}
 \norm{u-g(y_k)}_{L^\infty(\widetilde Q_k)}^{\delta_0}
 +d_k^{\,n+\frac{\delta_0}{1+\gamma}}
 \norm f_{L^\infty(\widetilde Q_k)}^{\frac{\delta_0}{1+\gamma}}
 \Bigr\}\le C\mathcal N^{\delta_0}d_k^n.\]
By {H\"older's inequality},
\[\int_{Q_k}|Du|^\delta\,dx\le |Q_k|^{1-\delta/\delta_0}
 \left(\int_{Q_k}|Du|^{\delta_0}\,dx\right)^{\delta/\delta_0}
 \le C\mathcal N^\delta d_k^n.
\]

Combining these estimates gives
\[
\left(\int_{\Omega_{1/12}}
 \bigl(|u|^\delta+|Du|^\delta+|D^2u|^\delta\bigr)\,dx
\right)^{1/\delta}
\le C\mathcal N,
\]
which is \eqref{eq:local-whitney}. The constant depends on
$n,\lambda,\Lambda,\gamma,\alpha,\delta_0,\delta$ and
$\norm\psi_{C^{1,\alpha}(B'_1)}$.
The derivatives are understood in the almost-everywhere sense specified
in the Notation paragraph.
\end{proof}

Theorem~\ref{thm:boundary-sobolev} now follows from
Theorem~\ref{thm:local-whitney} by a finite covering argument,
as in~\cite{LiLi23}.

\begin{proof}[Proof of Theorem~\ref{thm:boundary-sobolev}]
Since $T$ is $C^{1,\alpha}$, for each $x\in T$ there exist
$0<r_x\le1$ and a $C^{1,\alpha}$ function $\psi_x$ such that
$\partial\Omega\cap B_{r_x}(x)\subset T$ and, in suitable orthogonal coordinates,
\[
\begin{aligned}
\Omega_{r_x}(x)
 &=\{z\in B_{r_x}(x):z_n>\psi_x(z')\},\\
\partial\Omega\cap B_{r_x}(x)
 &=\{z\in B_{r_x}(x):z_n=\psi_x(z')\}.
\end{aligned}
\]
The set $T\cap\overline{\Omega'}=\partial\Omega\cap\overline{\Omega'}$
is compact. Thus finitely many balls $B_{r_i/12}(x_i)$,
$i=1,\ldots,N$, with $x_i\in T$ and $r_i=r_{x_i}$, cover this set.

Let $0<\delta\le\delta_0<1$, where $\delta_0$ is the interior exponent
in Corollary~\ref{lem:lili}.
In the coordinates of the $i$th boundary chart, the rescaling
$v_i(y)=r_i^{-2}u(x_i+r_i y)$ changes the right-hand side to
$r_i^{-\gamma}f(x_i+r_i y)$ and preserves the ellipticity constants.
Applying Theorem~\ref{thm:local-whitney} in each rescaled chart gives
\[
\begin{aligned}
&r_i^{-2-n/\delta}
 \norm u_{L^\delta(\Omega_{r_i/12}(x_i))}
 +r_i^{-1-n/\delta}
 \norm{Du}_{L^\delta(\Omega_{r_i/12}(x_i))}
 +r_i^{-n/\delta}
 \norm{D^2u}_{L^\delta(\Omega_{r_i/12}(x_i))}\\
&\le C_i\Bigl[
 r_i^{-2}\norm u_{L^\infty(\Omega_{r_i}(x_i))}
 +r_i^{-2}\norm g_{C^{1,\alpha}(T)}
 +r_i^{-\gamma/(1+\gamma)}
  \norm f_{L^\infty(\Omega_{r_i}(x_i))}^{1/(1+\gamma)}
 \Bigr].
\end{aligned}
\]
Here $C_i$ depends on the structural parameters and the corresponding
$C^{1,\alpha}$ boundary chart.

The remaining set
\[
K=\overline{\Omega'}\setminus\bigcup_{i=1}^N B_{r_i/12}(x_i)
\]
is a compact subset of $\Omega$.
Cover $K$ by finitely many balls $B_{s_j/2}(z_j)$ with
$B_{s_j}(z_j)\Subset\Omega$.
Corollary~\ref{lem:lili}, scaling and H\"older's inequality give
\[
\int_K\bigl(|u|^\delta+|Du|^\delta+|D^2u|^\delta\bigr)\,dx
\le C\left(\norm u_{L^\infty(\Omega)}
+\norm f_{L^\infty(\Omega)}^{1/(1+\gamma)}\right)^\delta.
\]
Summing the corresponding integrals over the boundary neighborhoods
and the interior balls proves \eqref{eq:boundary-sobolev}.
The finite covers and their radii depend only on $T,\Omega'$ and $\Omega$,
so the constant has the stated dependence.
\end{proof}

\section*{Acknowledgement}
This work was supported by National Natural Science Foundation of China (No. 12401257) and the Natural Science Foundation of Shandong Province, China (No. ZR2024QA045).

\section*{Data availability}
Data will be made available on request.
\section*{Declarations}	
The authors declare {that} they have no financial interests.

\end{document}